\documentclass[10pt,twosided,b4paper]{amsart}
\usepackage{epsfig}

\title[Decomposition of local representations]{Irreducible decomposition for local representations of quantum Teichm\"uller space}
\author{J\'{e}r\'{e}my Toulisse}
\address{
Department of Mathematics \\
University of Southern Califonia \\
3620 S. Vermont Avenue, KAP 104 \\
Los Angeles, CA 90089-2532}
\email{toulisse@usc.edu}
\date{\today}

\newcommand{\spann}{\text{span}}

\newcommand{\Id}{\text{Id}}
\newcommand{\End}{\text{End}}
\newcommand{\U}{\mathcal{U}(N)}
\newcommand{\p}{\textbf{p}}
\newcommand{\T}[1]{\mathcal{T}_q(#1)}
\newcommand{\TT}[1]{\hat{\mathcal{T}}_q(#1)}
\newcommand{\Sk}[1]{\mathcal{S}_A(#1)}

\newcommand{\Z}[1]{\mathcal{Z}_\omega(#1)}
\newcommand{\W}{\mathcal{W}(\tau_\lambda,\mathbb{Z})}

\theoremstyle{definition}
\newtheorem{Def}{Definition}[section]
\theoremstyle{plain}
\newtheorem{theo}[Def]{Theorem}
\newtheorem*{maintheo}{Main Theorem}

\newtheorem{prop}[Def]{Proposition}

\newtheorem{lemma}[Def]{Lemma}
\theoremstyle{remark}
\newtheorem{rem}[Def]{Remark}
\newtheorem{ex}[Def]{Example}

\begin{document}

\maketitle

\begin{abstract}
We give an irreducible decomposition of the so-called local representations \cite{baibonahon} of the quantum Teichm\"uller space $\T\Sigma$ where $\Sigma$ is a punctured surface of genus $g>0$ and $q$ is a $N$-th root of unity with $N$ odd. As an application, we construct  a family of representations of the Kauffman bracket skein algebra of the closed surface $\overline\Sigma$.
\end{abstract}

\tableofcontents

\section{Introduction}

Let $\Sigma$ be the surface obtained by removing $s>0$ points $v_1,...,v_s$ from the closed oriented surface $\overline{\Sigma}$ of genus $g>0$. Denote by $\mathcal{T}(\Sigma)$ the Teichm\"uller space of $\Sigma$, that is the moduli space of complete hyperbolic metrics on $\Sigma$. Given $\lambda$ an ideal triangulation of $\Sigma$ (that is a triangulation of the closed surface $\overline\Sigma$ whose vertices are exactly the $v_i$), Thurston \cite{thurston} constructed a parametrization of $\mathcal{T}(\Sigma)$ by associating a strictly positive real number to each edge $\lambda_i$ of the ideal triangulation, $i\in\{1,...,n\}$ (where $n=6g-6+3s$ is the number of edges of $\lambda$). These coordinates are called \textit{shear coordinates} associated to $\lambda$. In this coordinates system, the coefficients of the Weil-Petersson form on $\mathcal{T}(\Sigma)$ depend only on the combinatoric of $\lambda$ and are easy to compute.

For a parameter $q\in\mathbb{C}^*$, Chekhov and Fock \cite{fock1999quantum} defined the so-called \textit{quantum Teichm\"uller space} $\T\Sigma$ of $\Sigma$, which is a deformation quantization of the Poisson algebra of a certain class of functions over $\mathcal{T}(\Sigma)$ (see also Kashaev \cite{kashaev1998quantization} for a slightly different version). This algebraic object is obtained by gluing together a collection of non-commutative algebras $\T\lambda$, called \textit{Chekhov-Fock algebras}, canonically associated to each ideal triangulation of $\Sigma$. A representation of $\T\Sigma$ is then a family of representations $\{\rho_\lambda : \T\lambda \to \End(V) \}_{\lambda\in \Lambda(\Sigma)}$, where $\Lambda(\Sigma)$ is the space of all ideal triangulations of $\Sigma$, and $\rho_\lambda$ and $\rho_{\lambda'}$ satisfy compatibility conditions whenever $\lambda\neq \lambda'$. For $\lambda\in\Lambda(\Sigma)$, the representation $\rho_\lambda$ is an avatar of the representation of $\T\Sigma$ and carries almost all the information.

When $q$ is a primitive $N$-th root of unity, $\T\lambda$ admits finite-dimensional representations. In this paper, we will consider $N$ odd. The irreducible representations of $\T\lambda$ have been studied in \cite{bonahonliu}. In particular, they show that an irreducible representation of $\T\lambda$ is classified (up to isomorphism) by a weight $x_i\in\mathbb{C}^*$ assigned to each edge $\lambda_i$, a choice of $N$-th root $p_j=\left(x_1^{k_{j_1}}...x_n^{k_{j_n}}\right)^{1/N}$ associated to each puncture $v_j$ (where $k_{j_i}$ is the number of times a small simple loop around $v_j$ intersects $\lambda_i$) and a $N$-th root $c=(x_1...x_n)^{1/N}$. Such a representation has dimension $N^{3g-3+s}$.

In \cite{baibonahon}, the authors introduced another type of representations of $\T\lambda$, called \textit{local representations}, which are well behaved under cut and paste. A local representation of $\T\lambda$ is defined by a an embedding into the tensorial product of \textit{triangle algebras} (see definitions below). Isomorphism classes of local representations of $\T\lambda$ are classified by a weight $x_i\in\mathbb{C}^*$ associated to each edge $\lambda_i$ and a choice of $N$-th root $c=(x_1...x_n)^{1/N}$. Such a representation has dimension $N^{4g-4+2s}$.

It follows that a local representation of $\T\lambda$ is not irreducible. In this paper, we address the question of the decomposition of a local representation into its irreducible components. In particular, we prove the following result:

\begin{maintheo}
Let $\lambda$ be an ideal triangulation of $\Sigma$ and $\rho$ be a local representation of $\T\lambda$ classified by weight $x_j \in \mathbb{C}^*$ associated to each edge $\lambda_j$ and a choice of $N$-th root $c=(x_1...x_n)^{1/N}$. We have the following decomposition:
$$\rho=\bigoplus_{i\in \mathcal{I}}\rho^{(i)}.$$
Here, $\rho^{(i)}$ is an irreducible representation classified by the same $x_j$, a $N$-th root $p_j^{(i)}=(x_1^{k_{j_1}}...x_n^{k_{j_n}})^{1/N}$ associated to each puncture, and the same $c$. Moreover, given a choice of $N$-th root $p_j=(x_1^{k_{j_1}}...x_n^{k_{j_n}})^{1/N}$ for each puncture, there exists exactly $N^g$ elements $i\in\mathcal{I}$ with $p^{(i)}_j=p_j$ for all $j\in\{1,...,s\}$.
\end{maintheo}

It is proved by Bai in \cite{bai} that the intertwining operators associated to two isomorphic representations $\rho: \T\lambda \longrightarrow \End(V)$ and $\rho' : \T{\lambda'} \longrightarrow \End(V')$ correspond, in the case of the square, to the $6j$-symbols defined by Kashaev in \cite{kashaevlinkinvariant}. These $6j$-symbols relate hyperbolic geometry and quantum invariants and gave birth to the famous Volume Conjecture (see \cite{murakami} for an overview).

In particular, Baseilhac and Benedetti \cite{QHG} used these $6j$-symbols to construct a $(2+1)$ Topological Quantum Field Theory (TQFT) on manifolds with $\text{PSL}(2,\mathbb{C})$-character. Our result thus provides a decomposition of the vector spaces arising in the TQFT. 

\medskip

As an application, we adapt the construction of \cite{bonahonwongIII} to local representations of the balanced Chekhov-Fock algebra and obtain a family of representations of the Kauffman bracket skein algebra $\Sk{\overline\Sigma}$ of the closed surface $\overline\Sigma$. The vector space associated to this family of representations is canonically associated to an ideal triangulation $\lambda$. In particular, it makes the computations very explicit. It also behaves well under cut and paste.

\medskip

In the first section, we recall the definition of the Chekhov-Fock algebra, the quantum Teichm\"uller space, the triangle algebra and the local representations.
In Section \ref{mainsection}, we prove the Main Theorem. Finally, in Section \ref{consequences}, we explain the connections between quantum Teichm\"uller theory, skein theory and construct a new family of representations of $\Sk{\overline\Sigma}$.

\textbf{Acknowledgment :} \thanks{I would like to thank Francis Bonahon for valuable discussions on the subject. This research was partially supported by the grant DMS-1406559 from the U.S. National Science Foundation. The author gratefully acknowledges support from the NSF grants DMS-1107452, 1107263 and 1107367 ``RNMS: GEometric structures And Representation varieties'' (the GEAR Network).}

\section{Chekhov-Fock algebra and representations of $\T\Sigma$}

In this section, we define the Chekhov-Fock algebra $\T\lambda$ associated to an ideal triangulation $\lambda$, describe its representations and introduce the quantum Teichm\"uller space. Most results come from \cite{bonahonliu} and \cite{baibonahon}.

In all this paper, given an integer $n\in\mathbb{N}$, set $\mathbb{Z}_n:=\mathbb{Z}/n\mathbb{Z}$ and denote by $\mathcal{U}(N)$ the group of $N$-th root of unity.

\subsection{The Chekhov-Fock algebra}\label{chekhovfockalgebra} In this subsection, $q$ is a formal parameter and $\Sigma$ is allowed to have boundary components with punctures on the boundary (and every boundary component has at least one puncture). 

Let $\lambda$ be an ideal triangulation of $\Sigma$. We denote by $\lambda_1,...,\lambda_n$ the edges of $\lambda$. The \textit{Fock's matrix associated to $\lambda$} is the skew-symmetric matrix $\eta_\lambda=(\sigma_{ij})_{i,j}\in \mathcal{M}_n(\mathbb{Z})$ defined by
$$\sigma_{ij}=a_{ij}-a_{ji}$$ 
where $a_{ij}$ is the number of angular sector delimited by $\lambda_i$ and $\lambda_j$ in the faces of $\lambda$ with $\lambda_i$ coming before $\lambda_j$ counterclockwise.

\begin{Def}
The \textit{Chekhov-Fock algebra of $\lambda$} is the algebra $\T\lambda$ freely generated by the elements $X_i^{\pm 1},~i\in\{1,...,n\}$ subject to the relations
$$X_iX_j= q^{2\sigma_{ij}}X_jX_i.$$
\end{Def}

The following example is of main importance.

\begin{ex}\label{trianglealgebra}
Let $T$ be a disk with three punctures $v_1,v_2,v_3$ on the boundary. The boundary arcs between the punctures provides a natural triangulation $\lambda$ of $T$ (see Figure \ref{T}).
\begin{figure}[!h] 
\begin{center}
\includegraphics[height=3.5cm]{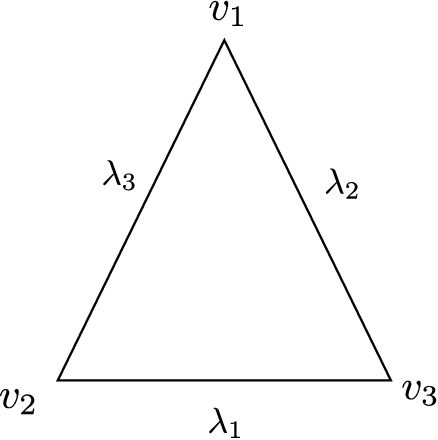}
\end{center}
\caption{The triangle $T$.} 
\label{T}
\end{figure}

The \textit{triangle algebra} is $\mathcal{T}:= \T\lambda$. It is generated by $X^{\pm 1}_i,~i=1,2,3$ subject to the relations
$$X_iX_{i+1}=q^2 X_{i+1}X_i,~i\in\mathbb{Z}_3.$$ 
\end{ex}
The algebraic structure of the Chekhov-Fock algebra is fairly simple. In particular, it is a quantum torus \cite{goodearl}.

Given a monomial $X\in\T\lambda$ composed of a product of $X_i^{k_i}$ for a multi-index $\textbf{k}=(k_1,...,k_n)\in \mathbb{Z}^n$, we define the \textit{Weyl ordering} of $X$ to be the monomial $[X]$ defined by
$$[X]:=q^{-\sum_{i<j}\sigma_{ij}}X_1^{k_1}...X_n^{k_n}.$$
The advantage of the Weyl ordering is its independence with respect to the order of the terms. In particular, for any permutation $\sigma: \{1,...n\} \to \{1,...,n\}$, we have
$$\left[X_1^{k_1}...X_n^{k_n}\right]=\left[X_{\sigma(1)}^{k_{\sigma(1)}}...X_{\sigma(n)}^{k_{\sigma(n)}}\right].$$
For $\textbf{k}=(k_1,...,k_n)\in\mathbb{Z}^n$ a multi-index, we define $X_\textbf{k}:=\left[ X_1^{k_1}...X_n^{k_n}\right]\in\T\lambda$.

\subsection{Finite dimensional representations of $\T\lambda$}\label{representationsChekhovFock} When the parameter $q$ is a root of unity, the structure of the Chekhov-Fock algebra is drastically different. In particular, $\T\lambda$ admits finite dimensional representations that we describe here.

In this subsection, $q\in\mathbb{C}^*$ is a primitive $N$-th root of unity with $N$ odd, $\Sigma$ has no boundary component and $\lambda$ is an ideal triangulation of $\Sigma$ with edges labeled $\lambda_1,...,\lambda_n$.

\begin{Def}\label{defpunctureinvariant}
For each puncture $v_j$, the \textit{puncture invariant} $P_j\in \T\lambda$ is the monomial associated to the multi-index $\textbf{k}_j=(k_{j_1},...,k_{j_n})\in\mathbb{N}^n$ where $k_{j_i}$ is the minimum number of intersections between the edge $\lambda_i$ and a closed simple loop around $v_j$.
\end{Def}

The puncture invariants are of main importance in the representation theory of the Chekhov-Fock algebra. In particular, we have the following:

\begin{prop}[\cite{bonahonliu}, Proposition 15]\label{center} The center of $\T\lambda$ is generated by:
\begin{itemize}
\item $X_i^N$ for each $i\in\{1,...,n\}$.
\item The \textit{puncture invariant} $P_j$ associated to each puncture $v_j\in\{v_1,...,v_s\}$.
\item The element $H:=\left[ X_1...X_n\right]$.
\end{itemize}
\end{prop}

Note that $[P_1...P_s]=[H^2]$.

A \textit{representation of $\T\lambda$} is a morphism $\rho: \T\lambda \longrightarrow \End(V)$ where $V$ is a vector space. Such a representation is \textit{finite dimensional} when $V$ is finite dimensional and $\rho$ is called \textit{irreducible} when the is no proper linear subspace $W\subset V$ preserved by $\rho\big(\T\lambda\big)$. Two representations $\rho: \T\lambda \longrightarrow \End(V)$ and $\rho': \T\lambda \longrightarrow \End(V')$ are \textit{isomorphic} if there exists a linear isomorphism $L: V \longrightarrow V'$ such that
$$\rho'(X) = L\circ \rho(X)\circ L^{-1},~\forall X\in\T\lambda.$$

\begin{theo}[\cite{bonahonliu}, Theorem 20-21]
Up to isomorphism, any irreducible representation 
$$\rho: \T\lambda\longrightarrow \End(V)$$
is determined by its restriction to the center of $\T\lambda$ and is classified by a non-zero complex number $x_i$ associated to each edges $\lambda_i$, for each puncture $v_j$, a choice of a $N$-th root $p_j=\left(x_1^{k_{j_1}}...x_n^{k_{j_n}}\right)^{1/N}$ (where $k_{j_k}\in\{1,2\}$ are defined in Definition \ref{defpunctureinvariant})  and a choice of a square root $c=(p_0...p_s)^{1/2}$.

Such a representation is characterized by:
\begin{itemize}
\item $\rho\left(X_i^N\right)=x_i \Id_V,$
\item $\rho(P_j)=p_j \Id_V,$
\item $\rho(H)=c \Id_V.$
\end{itemize}
Moreover, such a representation has dimension $N^{3g-3+s}$.
\end{theo}

Let us come back to Example \ref{trianglealgebra}. Remind that the triangle algebra $\mathcal{T}$ is the algebra generated by $X_i^{\pm 1},~i\in\mathbb{Z}_3$ with relations $X_iX_{i+1}=q^2X_{i+1}X_i$.

The center of $\mathcal{T}$ is generated by $X_i^N$ and $H=q^{-1}X_1X_2X_3$. One easily checks that irreducible representations of $\mathcal{T}$ have dimension $N$ and are classified (up to isomorphism) by a choice of weight $x_i\in\mathbb{C}^*$ associated to each edge $\lambda_i$ and a central charge, that is a choice of a $N$-th root $c=(x_1x_2x_3)^{1/N}$ (see \cite[Lemma 2]{baibonahon} for more details).

More precisely, let $V$ the complex vector space generated by $\{e_1,...,e_N\}$ and $\rho$ an irreducible representation of $\mathcal{T}$ classified by $x_1,x_2,x_3\in\mathbb{C}^*$ and $c=(x_1x_2x_3)^{1/N}$. Up to isomorphism, the action of $\mathcal{T}$ on $V$ defined by $\rho$ is given by:
$$\left\{
    \begin{array}{ll}
        \rho(X_1)e_i =  & x'_1q^{2i}e_i \\
        \rho(X_2)e_i = & x'_2e_{i+1} \\
        \rho(X_3)e_i = & x'_3 q^{1-2i}e_{i-1}
    \end{array}
\right.$$
where $x'_i$ is an $N$-th root of $x_i$ such that $x'_1x'_2x'_3=c$. Note that, up to isomorphism, $\rho$ is independent of the choice of the $N$-th root $x'_i$ with $x'_1x'_2x'_3=c$. 

The following lemma will be useful in the next section. Remind that $\U$ is the group of $N$-th root of unity.

\begin{lemma}\label{vp}
Let $\rho: \mathcal{T} \longrightarrow \End(V)$ be the representation of the triangle algebra classified by $x_1=x_2=x_3=1$ and $c\in \U$. Ror each $i\in \mathbb{Z}_3$ and $N$-th root $\zeta\in \U$, the eigenspace of $\rho(X_i)$ of eigenvalue $\zeta$ is one-dimensional.
\end{lemma}

\begin{proof}
We use the explicit form of the representation $\rho$ in $V=\spann\{e_1,...,e_N\}$ described above. Set $x_1'=x'_2=1,~x'_3=c$ and $\zeta=q^{2k}$. 

For $i=1$, one checks that the eigenspace of $\rho(X_1)$ associated to $h$ is generated by $e_k$.

For $i=2$, the vector $\alpha_k:=\sum_{i\in\mathbb{Z}_N} q^{-2ki}e_i$ satisfies $\rho(X_2)\alpha_k=q^{2k}\alpha_k$ and $\{\alpha_1,...,\alpha_k\}$ form a basis of $V$. So the eigenspace of $\rho(X_2)$ associated to the eigenvalue $\zeta$ is generated by $\alpha_k$.

For $i=3$, we use the fact that $\rho(q^{-1}X_1X_2X_3)=c\Id_V$, where $c\in\mathcal{U}(N)$.
\end{proof}

An ideal triangulation of $\Sigma$ is composed by $m$ faces $T_1,...,T_m$. Each face $T_j$ determines a triangle algebra $\mathcal{T}_j$ whose generators are associated to the three edges of $T_j$. It provides a canonical embedding $\iota$ of $\T\lambda$ into $\mathcal{T}_1\otimes...\otimes\mathcal{T}_m$ defined on the generators as follow:
\begin{itemize}
\item $\iota(X_i)= X_{j_i}\otimes X_{k_i}$ if $\lambda_i$ belongs to two distinct triangles $T_j$ and $T_k$ and $X_{j_i}\in\mathcal{T}_j$, $X_{k_i}\in\mathcal{T}_k$ are the generators associated to the edge $\lambda_i\in T_j$ and $\lambda_i\in T_k$ respectively.
\item $\iota(X_i)=[X_{j_{i_1}}X_{j_{i_2}}]$ if $\lambda_i$ corresponds to two sides of the same face $T_j$ and $X_{j_{i_1}},X_{i_{j_2}}\in\mathcal{T}_j$ are the associated generators.
\end{itemize}

\begin{Def}
A \textit{local representation} of $\T\lambda$ is a representation which factorizes as $(\rho_1\otimes...\otimes\rho_m)\circ \iota$ where $\rho_i : \mathcal{T}_i \to V_i$ is an irreducible representation of the triangle algebra $\mathcal{T}_i$ and $\iota: \T\lambda \longrightarrow \mathcal{T}_1\otimes...\otimes \mathcal{T}_m$ is defined as above.
\end{Def}
Note that in particular, a local representation has dimension $N^m$ where $m=4g-4+2s$ is the number of faces of the triangulation.

Local representations were first introduced in \cite{baibonahon}. In particular, they proved the following:

\begin{theo}[\cite{baibonahon}, Proposition 6]
Up to isomorphism, a local representation 
$$\rho: \T\lambda \longrightarrow \End(V)$$ 
is classified by a non-zero complex number $x_i$ associated to the edge $\lambda_i$ and a choice of a $N$-th root $c=(x_1...x_n)^{1/N}$. Such a representation satisfies:
\begin{itemize}
\item $\rho(X_i^N)=x_i \Id_V,$
\item $\rho(H)=c \Id_V.$
\end{itemize}
\end{theo}

Local representations have certain advantages over irreducible representations. First of all, these representations behave very well under cut and paste, so one can used them to construct invariant of 3-manifolds (see \cite{QHG}). Also, the vector space associated to a local representation decomposes as a tensor product of vector spaces and each generator $X_i\in\T\lambda$ associated to an edge $\lambda_i$ only acts on the vector spaces associated to triangle adjacent to the edge $\lambda_i$ (that is why these representations are called local). 

\subsection{The quantum Teichm\"uller spaces and its representations}\label{repteich}

The quantum Teichm\"uller space is obtained by gluing together a family of (division algebra of) Chekhov-Fock algebras indexed by the set of ideal triangulations of $\Sigma$.

\medskip

\noindent\textbf{The simplex of ideal triangulations:} Let $\Lambda(\Sigma)$ be the set of ideal triangulations of $\Sigma$. We say that two triangulations $\lambda,\lambda'\in\Lambda(\Sigma)$ differ by a flip is $\lambda$ and $\lambda'$ coincide everywhere except in a square made of two adjacent triangles where they differ as in Figure \ref{flip}.

\begin{figure}[!h] 
\begin{center}
\includegraphics[height=1.5cm]{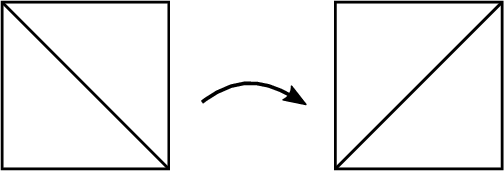}
\end{center}
\caption{Flip of triangulation.} 
\label{flip}
\end{figure}

The \textit{graph of ideal triangulations} is the graph whose set of vertices is $\Lambda(\Sigma)$ and two vertices $\lambda,\lambda'\in\Lambda(\Sigma)$ are connected by an edge if and only if $\lambda$ and $\lambda'$ differ by a flip.

The \textit{simplex of ideal triangulations} is obtained from the graph of ideal triangulations by gluing a 2-simplex on each cycle corresponding to the pentagon relation (see Figure \ref{pentagon}).

\begin{figure}[!h] 
\begin{center}
\includegraphics[height=6cm]{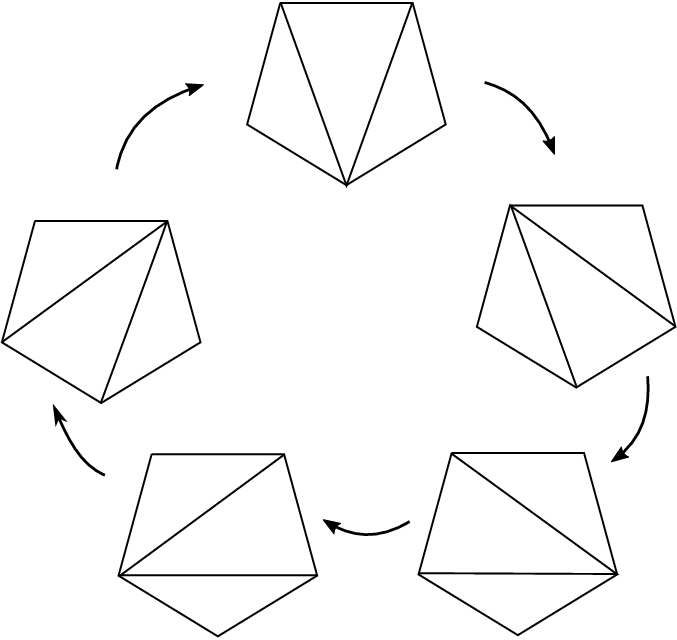}
\end{center}
\caption{Pentagon relation.} 
\label{pentagon}
\end{figure}

\begin{prop}[Penner \cite{penner}]
The simplex of ideal triangulations is connected and simply connected. Namely, any two different ideal triangulations are connected by a sequence of flips and two sequences between two ideal triangulations differ by a sequence of pentagon relations.
\end{prop}

\noindent\textbf{Coordinates change:} The Chekhov-Fock algebra $\T\lambda$ associated to an ideal triangulation $\lambda\in\Lambda(\Sigma)$ satisfies the Ore condition (see \cite{goodearl}). In particular, $\T\lambda$ has a well-defined division algebra $\TT\lambda$ consisting of rational fractions satisfying some non-commutativity relations. 

Let $\lambda,\lambda'\in\Lambda(\Sigma)$ be two ideal triangulations related by a flip. In \cite{fock1999quantum}, Chekhov and Fock constructed coordinates change isomorphisms
$$\Psi^q_{\lambda\lambda'}: \TT{\lambda'}\longrightarrow \TT\lambda.$$
These coordinates change satisfy the pentagon relation. In particular, using the result of Penner, they extend uniquely to coordinates change $\Psi^q_{\lambda\lambda'}: \TT{\lambda'}\longrightarrow \TT\lambda$ for any two different ideal triangulations $\lambda,\lambda'\in\Lambda(\Sigma)$.

It was proved in \cite{Liu} that these coordinates change are the unique ones satisfying some natural relations, as for instance $\Psi^q_{\lambda\lambda''}=\Psi^q_{\lambda\lambda'}\circ\Psi^q_{\lambda'\lambda''}$ for any $\lambda,\lambda'\lambda''\in\Lambda(\Sigma)$. Moreover, when $q=1$, these maps reduce to the classical coordinates change in Teichm\"uller theory (see \cite{Liu} for more details).

\begin{Def}
The \textit{quantum Teichm\"uller space of $\Sigma$} is defined by
$$\T\Sigma:=\bigsqcup_{\lambda\in\Lambda(\Sigma)} \hat{\mathcal{T}}_q(\lambda) / \sim,$$
where $x_\lambda\in \TT\lambda \sim x_{\lambda'}\in\TT{\lambda'}$ if and only if $x_\lambda = \Psi_{\lambda\lambda'}^q(x_{\lambda'})$.
\end{Def}
Note that, as each coordinates change $\Psi^q_{\lambda\lambda'}$ is an algebra isomorphism, $\T\Sigma$ inherits an algebra structure, and the $\TT\lambda$ can be thought as ``global coordinates'' on $\T\Sigma$.

\medskip

\noindent\textbf{Representations:} A natural definition for a finite dimensional representation of $\T\Sigma$ would be a family of finite dimensional representations 
$$\{\rho_\lambda: \hat{\mathcal{T}}_q(\lambda)\longrightarrow \End(V_\lambda)\}_{\lambda\in\Lambda(\Sigma)}$$ such that for each pair $\lambda,\lambda'\in\Lambda(\Sigma)$, the representations $\rho_{\lambda'}$ and $\rho_\lambda\circ \Psi^q_{\lambda,\lambda'}$ of $\TT{\lambda'}$ are isomorphic (as representations). 

However, as pointed out in \cite[Section 4.2]{baibonahon}, when $V_\lambda$ is finite dimensional, there is no morphism $\TT\lambda \longrightarrow \text{End}(V_\lambda)$. In fact, $\TT\lambda$ is infinite-dimensional as a vector space while $\End{V_\lambda}$ is finite dimensional and so, such a homomorphism $\rho_\lambda$ would have non-zero kernel. In particular, there would exists elements $x\in\hat{\mathcal{T}}_q(\lambda)$ such that $\rho_\lambda(x)=0$ and so, $\rho_\lambda(x^{-1})$ would make no sense.

It motivates the following definition:

\begin{Def}
A \textit{local representation} (respectively an \textit{irreducible representation}) of $\T\Sigma$ as a family of representation 
$$\{\rho_\lambda : \T\lambda\longrightarrow \End(V_\lambda)\}_{\lambda\in\Lambda(\Sigma)}$$
such that for each $\lambda,\lambda'\in\Lambda(\Sigma)$, $\rho_\lambda$ is a local representation (respectively an irreducible representation) of $\T\lambda$, and $\rho_{\lambda'}$ is isomorphic to $\rho_\lambda\circ \Psi^q_{\lambda\lambda'}$ whenever $\rho_\lambda\circ \Psi^q_{\lambda\lambda'}$ makes sense.

We say that $\rho_\lambda\circ \Psi^q_{\lambda\lambda'}$ makes sense, if for each Laurent polynomial $X'\in\T{\lambda'}$,
$$\Psi_{\lambda\lambda'}(X')=PQ^{-1}=Q'^{-1}P'\in\hat{\mathcal{T}}_q(\lambda),~\text{for } P,Q,P',Q' \in \T\lambda.$$
In that case, we define:
$$\rho_\lambda\circ\Psi_{\lambda\lambda'}(X'):=\rho_\lambda(P)\rho_\lambda(Q)^{-1}=\rho_\lambda(Q')^{-1}\rho_\lambda(P').$$
\end{Def}

We have the following:

\begin{prop}[\cite{baibonahon} Proposition 10]
Let $\lambda,\lambda'\in\Lambda(\Sigma)$ be two ideal triangulations of $\Sigma$. There exists a rational map
$$\varphi_{\lambda\lambda'}:\mathbb{C}^n\longrightarrow \mathbb{C}^n$$
such that a local representation $\rho': \T{\lambda'} \longrightarrow \End(V_{\lambda'})$ classified by $(x'_1,...,x'_n)$ and $c'=\big(x_1'...x_n'\big)^{1/N}$ is isomorphic to $\rho_\lambda\circ\Psi_{\lambda\lambda'}$ (whenever it makes sense) where $\rho_\lambda:\T\lambda \longrightarrow \End(V_\lambda)$ is a local representation classified by $(x_1,...,x_n)$ and $c=(x_1...x_n)^{1/N}$ if and only if  $c=c'$ and
$$(x'_1,...,x'_n)=\varphi_{\lambda\lambda'}(x_1,...,x_n).$$
\end{prop}

\begin{rem}
The analogue is also proved in \cite{bonahonliu} for irreducible representations. In particular, the rational maps $\varphi_{\lambda\lambda'}$ are the same.
\end{rem}

It turns out that the rational maps $\varphi_{\lambda\lambda'}$ correspond to the coordinates change of the shear-bend coordinates on the character variety $\chi\big(\Sigma,\text{SL}(2,\mathbb{C})\big).$

As a result, isomorphism classes of local (respectively irreducible) representations of $\T\Sigma$ are classified, up to finitely man choices, by the character variety $\chi\big(\Sigma,\text{SL}(2,\mathbb{C})\big)$ (see \cite{bonahonliu} for more details).
\section{Decomposition of local representations}\label{mainsection}

In this section, we prove Main Theorem. Let $\rho: \T\lambda \longrightarrow \End(V)$ be the local representation classified by a non-zero complex number $x_i$ associated to each edge and central charge $c$.  Given a puncture invariant $P_j=[X_1^{k_{j_1}}...X_n^{k_{j_n}}]$ (see Proposition \ref{center}) associated to the puncture $v_j$, the representation $\rho$ satisfies
$$\rho(P_j^N)=x_1^{k_{j_1}}...x_n^{k_{j_n}} \Id_V.$$
It follows that in $p_j$ is an eigenvalue of $P_j$, $p_j^N=x_1^{k_{j_1}}...x_n^{k_{j_n}}$.
\medskip

\noindent\textbf{Notations:} 
\begin{itemize}
\item Given $p_j\in \mathbb{C}^*$ so that $p_j^N=x_1^{k_{j_1}}...x_n^{k_{j_n}}$, we denote
$$V_{p_j}(P_j):=\{x\in V,~\rho(P_j)x=p_j x\}$$
the associated eigenspace.
\item Given $\p=(p_1,...,p_s)$ so that for each $j$ $p_j^N=x_1^{k_{j_1}}...x_n^{k_{j_n}}$, set
$$V_\p:=\{x\in V,~\rho(P_j)x=p_jx,~j=1,...,s\}=\bigcap_{j=1}^s V_{p_j}(P_j).$$
\end{itemize}

The proof of Main Theorem will follow the next proposition:

\begin{prop}\label{mainprop}
There exists an ideal triangulation $\lambda_0\in \Lambda(\Sigma)$ so that for each $\p$ as above, $V_\p$ has dimension $N^{4g-3+s}$.
\end{prop}

\begin{proof}
Note that the dimension of $V_\p$ does not depend on the numbers $x_i\in \mathbb{C}^*$ characterizing $\rho$. In this proof, we will consider all the $x_i$ equal to $1$, which implies that the eigenvalues of $\rho(P_i)$ are root of unity.

Consider the one punctured surface $\Sigma':=\Sigma\cup \{v_1,...,v_{s-1}\}$. As $g>0$, there exists an ideal triangulation $\lambda'$ of $\Sigma'$. Let $T$ be a triangle of the triangulation $\lambda'$ and consider the triangulation of $T\setminus \{v_1,...,v_{s-1}\}$ as in Figure \ref{triangulationofT}.

\begin{figure}[!h] 
\begin{center}
\includegraphics[height=8cm]{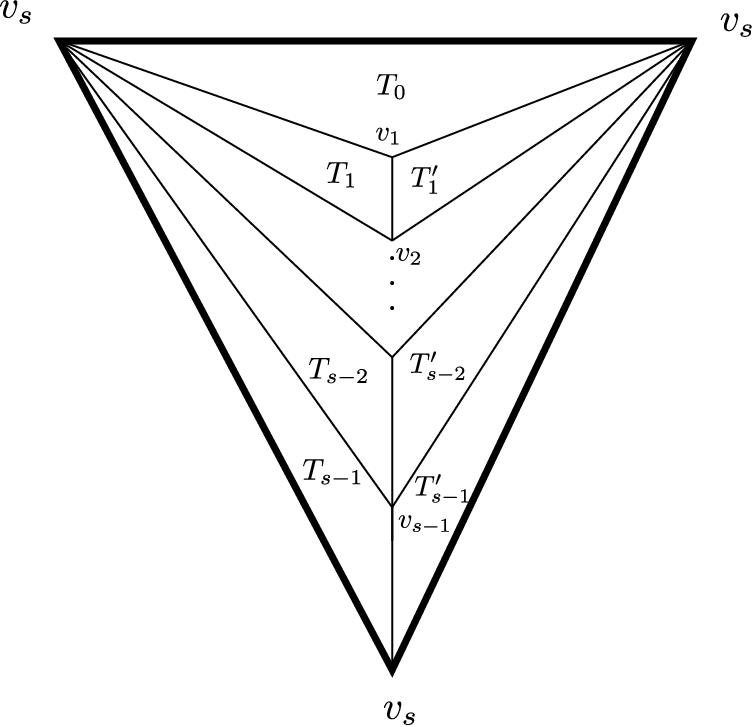}
\end{center}
\caption{Triangulation of $T\cup \{v_1,...,v_s\}$} 
\label{triangulationofT}
\end{figure}

The union of the triangulation $\lambda'$ and the one of $T$ gives an ideal triangulation $\lambda_0$ of $\Sigma$. 

Consider a local representation $\rho: \T{\lambda_0}\longrightarrow \End(V)$. The decomposition of the ideal triangulation $\lambda_0$ gives the following nice decomposition:
$$V=W\otimes W',$$
where $W'$ is the vector space corresponding to the triangles of the triangulation $\lambda'$ (except the triangle $T$), and $W$ corresponds to the triangles of $T$.

In particular, as the triangulation $\lambda'$ contains $4g-2$ triangles, $\dim(W')=N^{4g-3}$ (remember that we do not consider the vector space associated to $T$), and $\dim W= N^{2s-1}$.

The interest of the triangulation $\lambda_0$ is clear: the puncture invariant $P_i$ associated to the puncture $v_i\neq v_s$ acts as the identity on $W'$, so the eigenspaces of $\rho(P_i)$ has the form $E\otimes W'$ where $E\subset W$ is an eigenspace of the restriction of $\rho(P_i)$ on $W$. It motivates the following notations:

\medskip

\noindent\textbf{Notations:}
\begin{itemize}
\item The vector space $W$ decomposes as 
$$W=W^0\otimes...\otimes W^{s-1},$$
where $W^0$ is associated to $T_0$ and $W^j$ to $T_j$ and $T'_j$ for $j=1,...,s-1.$
\item Given a root of unity $p_k\in \U$, set
$$W^j_{p_k}(P_k):=\{x\in W^j,~\rho(P_k)x=p_kx\}.$$
\item For $\p=(p_1,...,p_{s-1})\in \U^{s-1}$, set
$$W^j_\p=\{x\in W^j,~\rho(P_k)x=p_kx,~k=1,...,s-1\}=\bigcap_k^{s-1} W^j_{p_k}(P_k).$$
\item Finally, set 
$$W_\p=\{x\in W,~\rho(P_k)x=p_kx,~k=1,...,s-1\}.$$
\end{itemize}

We have the following:

\begin{lemma}\label{mainlemma} Using the above notations, and given $\p\in \U^{s-1}$, we have the following:

\begin{itemize}
\item[1.] $
\dim W^0_{\p} = \left\{
    \begin{array}{ll}
        1 & \mbox{if } \p=(p_1,1,...,1)\\
        0 & \mbox{otherwise.}
    \end{array}
\right.
$
\item[2.]
For $j\in\{1,...,s-2\}$,
$$\dim W^j_{\p} = \left\{
    \begin{array}{ll}
        1 & \mbox{if } \p=(1,...,1,p_j,p_{j+1},1,...,1) \\
        0 & \mbox{otherwise.}
    \end{array}
\right.
$$
\item[3.] $
\dim W^{s-1}_{\textbf{h}} = \left\{
    \begin{array}{ll}
        N & \mbox{if } \p=(1,...,1,p_{s-1}) \\
        0 & \mbox{otherwise.}
    \end{array}
\right.
$
\end{itemize}
\end{lemma}
\begin{proof}
1. If $k\neq 1$, $v_k$ is not a vertex of $T_0$. It follows that $P_k$ acts on $W^0$ by the identity; so if $p_k\neq 1$, $W^0_{\p}=\{0\}$.

Now, if $p_k=1$ for all $k\neq 1$, then $W^0_{\p}$ is the eigenspace of the action on $W^0$ of the edge opposite to $v_1$. So by Lemma \ref{vp}, it is one dimensional.

\medskip

2. Fix $j\in\{1,...,s-2\}$. For $k\notin\{j,j+1\}$, $v_k$ is neither a vertex of $T_j$ nor of $T'_j$. So $P_j$ acts on $W^j$ as the identity. Hence, if $p_k\neq 1$, then $W^j_{\p}=\{0\}$.

Take $p_k=1$ for all $k\notin \{j,j+1\}$ and denote by $X^{\pm 1},Y^{\pm 1},Z^{\pm 1}$ (respectively ${X'}^{\pm 1},{Y'}^{\pm 1},{Z'}^{\pm 1}$) the generators of the triangle algebras $\mathcal{T}_j$ (respectively $\mathcal{T}'_j$) associated to the triangles $T_j$ (respectively ${T'}_j$) as in Figure \ref{T_j}. Set also $W^j=V^j\otimes V'^j$ where $V^j$ (respectively $V'^j$) is associated to $T_j$ (respectively $T'_j$).

\begin{figure}[!h] 
\begin{center}
\includegraphics[height=7cm]{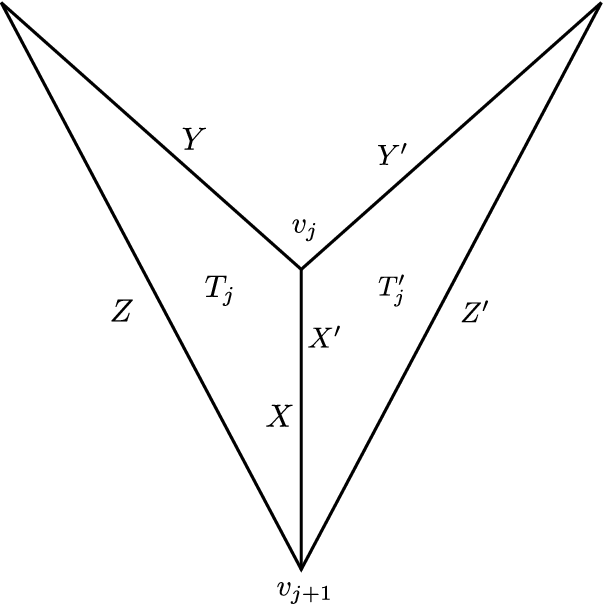}
\end{center}
\caption{The generators of $\mathcal{T}_j$ and $\mathcal{T}'_j$} 
\label{T_j}
\end{figure}

Denote by $c_j, c'_j\in\mathcal{U}(N)$ the central charges of the restriction of the representation to $\mathcal{T}_j$ and $\mathcal{T}'_j$ respectively. $\rho(P_j)$ acts on $V^j:=\spann\{e_0,...,e_{N-1}\}$ like $c_jZ^{-1}$ and on $V'^j=\spann\{e'_0,...,e'_{N-1}\}$ like $c'_jZ'^{-1}$. In the same way, $\rho(P_{j+1})$ acts on $V_j$ like $c_jY^{-1}$ and on $V'_j$ like $c'_jY'^{-1}$. 

Using the same action as in Section \ref{trianglealgebra} and writing $\left\{\begin{array}{l} c_j=q^p\\
c'_j=q^{p'}\end{array}\right.$, we get the following:

\begin{eqnarray*}
\rho(P_j) e_k & =q^{2k-1+p}e_{k+1} \\
\rho(P_j) e'_l & =q^{1-2l+p'}e_{l+1}
\end{eqnarray*}

It follows that the action of $P_j$ on $W^j$ is given by:
$$P_j\epsilon_{k,l}=q^{2(k-l)+p+p'}\epsilon_{k+1,l+1} \text{ where } \epsilon_{k,l}:=e_k\otimes e'_l\in W^j.$$
In the same way, one sees that the action of $P_{j+1}$ on $W^j$ is given by:
$$P_{j+1}\epsilon_{k,l}=q^{p+p'}\epsilon_{k-1,l-1}.$$
Now, for $m,n\in\mathbb{Z}_N$, set $\displaystyle{\alpha_{m,n}:=\sum_{k=0}^{N-1}q^{2km}\epsilon_{k,k+n}}$, an easy calculation shows that:
$$\left\{
    \begin{array}{lll}
        P_j\alpha_{m,n} & = & q^{-2(m+n)+p+p'}\alpha_{m,n} \\
        P_{j+1}\alpha_{m,n} & = & q^{2m+p+p'}\alpha_{m,n}.
    \end{array}
\right.$$
It follows that $\left\{ \alpha_{n,m},~n,m\in\mathbb{Z}_N\right\}$ is a base of $W^j$ and, for all $p_j,p_{j+1}\in \mathcal{U}(N)$, there exists a unique couple $(m,n)\in\mathbb{Z}_N^2$ with $p_j=q^{-2(m+n)+p+p'}$ and $p_{j+1}=q^{2m+p+p'}$.

So $\dim W^j_{\p}=1$ if and only if $p_k=1$ for all $k\neq j,j+1$.

\medskip

3. If $k\neq s-1$, $v_k$ is neither a vertex of $T_{s-1}$ nor $T'_{s-1}$, so if $p_k\neq 1$, $W^s_{\textbf{h}}=\{0\}$.

Suppose that $p_k=1$ for all $k\in\{1,...,s-2\}$, then 
$$W^{s-1}_{\textbf{h}}\supset\bigoplus_{p_ap_b=p_{s-1}} V^{s-1}_{p_a}(P_{s-1})\otimes V'^{s-1}_{p_b}(P_{s-1}),$$
where $V^s_{p_a}(P_{s-1})$ is the eigenspace associated to the eigenvalue $p_a$ of the action of $\rho(P_{s-1})$ on the vector space associated to the triangle $T_{s-1}$ and $V'^{s-1}_{p_b}(P_{s-1})$ is defined in an analogous way. 

The direct sum contains $N$ terms of dimension one, hence $\text{dim}(W^{s-1}_\p)\geq N$. But, we have 
$$\dim(W^{s-1})=N^2=\sum_{{\p}\in\mathcal{U}(N)^{s-1}}\dim(W^{s-1}_{\p})\geq N\times N.$$ 
So $W^{s-1}_{\p}$ has exactly dimension $N$.
\end{proof}

The proof of Proposition \ref{mainprop} follows from the following elementary remark:

\begin{rem}
For all $j\in \{0,...,s-1\}$, given $\p_j\in \U^{s-1}$ and a non-zero vector $x_j\in W^j_{\p_j}$, the vector $x_0\otimes...\otimes x_{s-1}$ is in $W_\p$ where $\p=\p_0\p_1...\p_{s-1}$.  
\end{rem}  

It follows that we have the following inclusion:

\begin{equation}\label{directsum}
W_\p\supset \bigoplus_{\p=\p_0...\p_{s-1}}W^0_{\p_0}\otimes...\otimes W^{s-1}_{\p_{s-1}}.
\end{equation}

Writing $\p_j=\left(p^{(j)}_0,...,p^{(j)}_{s-1}\right)$ and $\p=(p_1,...,p_s)$, one notes that from Lemma \ref{mainlemma}, the only non-zero terms in the direct sum of equation (\ref{directsum}) are the $\p_j$ satisfying
\begin{equation}\label{relations}
\left\{
    \begin{array}{ll}
        p^{(0)}_1p^{(1)}_1 =   p_1 \\
        p^{(1)}_2p_2^{(2)} =   p_2 \\
        \vdots \\
        p_{s-1}^{(s-2)}p_{s-1}^{(s-1)} =   p_{s-1}
    \end{array}
\right.
\end{equation}

There exists exactly $N^{s-1}$ different choices for the $\p_j$ satisfying relations (\ref{relations}). 

Moreover, for each choice of $\p_j$ satisfying (\ref{relations}), the vector space $W^0_{\p_0}\otimes...\otimes W^{s-1}_{\p_{s-1}}$ has dimension $N$. It follows that for each $\p\in \U^{s-1}$, $$\dim W_\p \geq N^s.$$

However, 
$$\dim W = N^{2s-1}=\sum_{\p\in \U^{s-1}} \dim W_\p,$$
so each $W_\p$ has exactly dimension $N^s$.

Finally, as the puncture invariants act as the identity on the vector space $W'$, the intersection of the eigenspaces of the $\rho(P_j)$ for all $j=1,...,s-1$ has the form $W_\p \otimes W'$ for some $\p\in\U^{s-1}$ and has dimension $N^{4g-3+s}$. But has the representation $\rho$ has fixed central charge $c$,
$$\rho([P_1P_2...P_s])=\rho([H^2])=c^2\Id_V.$$
It follows that the action of $\rho(P_s)$ on $V$ is easily expressed in function of the action of the $\rho(P_j)$ for $j=1,...,s-1$ and we get the result.
\end{proof}

Note that Proposition \ref{mainprop} implies the decomposition of Main Theorem for the triangulation $\lambda_0$. The dimension of the eigenspaces depend continuously on the $x_i$, in particular, we get the decomposition for all value of $x_i\in\mathbb{C}^*$.

In fact, consider the local representation
$$\rho: \T{\lambda_0} \longrightarrow \End(V)$$
classified by a non-zero complex number $x_i$ associated to each edge and central charge $c$. Let $\rho^{(i)}: \T{\lambda_0} \longrightarrow \End(V^{(i)})$ be an irreducible factor.

In particular,
$$\left\{\begin{array}{lllll}
\rho^{(i)}(X_i^N) & = & \rho (X_i^N)_{\vert V^{(i)}} & = & x_i\Id_{V^{(i)}} \\
\rho^{(i)}(H) & = & \rho (H)_{\vert V^{(i)}} & = & c\Id_{V^{(i)}},\end{array}\right.$$
so a necessary condition for $\rho^{(i)}$ to be an irreducible factor is to be classified by the same $x_i$ and same central charge $c$.

For each puncture $v_j$, denote $p_j^{(i)}$ the $N$-th root of $x_1^{k_{j_1}}...x_n^{k_{j_n}}$ so that
$$\rho^{(i)}(P_j)=p_j\Id_{V^{(i)}}.$$
It follows that $p_s^{(i)}=c^2 \left(p_1^{(i)}...p_{s-1}^{(i)}\right)^{-1}$ and $V^{(i)}\subset V_\p$ where $\p=\left(p_1^{(i)},...,p_{s-1}^{(i)}\right)$ with
$$V_\p=\left\{x\in V,~\rho(P_j)x=p_j^{(i)}x,~j=1,...,s-1\right\}.$$

Finally, as $\dim V_\p=N^{4g-3+s}$ and the dimension of an irreducible representation of $\T{\lambda_0}$ has dimension $N^{3g-3+s}$, $V_\p$ contains exactly $N^g$ irreducible factors classified by the same $x_i$, same central charge $c$ and $N$-the root $p_j^{(i)}$ associated to the puncture $v_j$.

\medskip

\noindent\textbf{Proof in the general case.}
Remind that, given another ideal triangulation $\lambda\in \Lambda(\Sigma)$, the ``transition maps'' $\varphi_{\lambda_0\lambda}: \mathbb{C}^n \longrightarrow \mathbb{C}^n$ defined in Section \ref{repteich} are rational, hence defined on a Zariski open set of $\mathbb{C}^n$.

Now, consider a local representation 
$$\rho: \T\lambda \longrightarrow \End(V_\lambda),$$
classified by a number $x_i\in \mathbb{C}^*$ associated to each edge and central charge $c$.

If there exists $(y_1,...,y_n)\in \mathbb{C}^n$ so that $\varphi_{\lambda_0\lambda}(y_1,...,y_n)=(x_1,...,x_n)$ (which is a generic condition), then it follows from Section \ref{repteich} that $\rho_\lambda$ is isomorphic to $\rho_{\lambda_0}: \T{\lambda_0} \longrightarrow \End(V_{\lambda_0})$. It means that there exists an isomorphism
$$L_{\lambda_0\lambda}: V_\lambda \longrightarrow V_{\lambda_0},$$
so that for each $X\in \T\lambda$ we have
$$\rho_{\lambda_0}(\Psi^q_{\lambda_0\lambda}(X))=L_{\lambda_0\lambda}\circ\rho_\lambda(X)\circ L_{\lambda_0\lambda}^{-1}.$$
Note that here $\Psi^q_{\lambda_0\lambda}: \hat{\mathcal{T}}_q(\lambda) \longrightarrow \hat{\mathcal{T}}_q(\lambda_0)$ are the coordinates change defined in Section \ref{repteich}.

As $\rho_{\lambda_0}$ is a local representation of $\T\lambda_0$, there exists an irreducible decomposition of $\rho_{\lambda_0}$ given by the decomposition $$\displaystyle{V_{\lambda_0}=\bigoplus_{i\in\mathcal{I}} V_{\lambda_0}^i}.$$
In particular, each $i\in \mathcal{I}$, $V^i_{\lambda_0}$ is stable by $\rho_{\lambda_0}$ and has dimension $N^{3g-3+s}$.

For each $i\in \mathcal{I}$, set $V_\lambda^i:= L^{(-1)}_{\lambda_0\lambda}(V^i_{\lambda_0})$. One easily gets that each $V_\lambda^i$ is stable by $\rho_\lambda(\T\lambda)$, has dimension $N^{3g-3+s}$. So we get a decomposition of $\rho_\lambda$ into irreducible factors given by the decomposition
$$V_\lambda= \bigoplus_{i\in\mathcal{I}} V_\lambda^i.$$

Finally, if $\rho_\lambda$ is classified by the parameters $(x_1,...,x_n)$ which are not in the image of $\varphi_{\lambda_0\lambda}$, one can deform continuously $(x_1,...,x_n)$ to get the previous decomposition and, as the decomposition does not depend of the parameters, get the result for $\rho_\lambda$.  

\section{Representations of the Skein algebra}\label{consequences}

In this section, we use Main Theorem to construct a nice family of representation of the Kauffman bracket skein algebra $\Sk{\overline{\Sigma}}$ of the closed surface $\overline\Sigma=\Sigma \cup \{v_1,...,v_s\}$. This is done by adapting the construction of Bonahon and Wong \cite{bonahonwongIII} to the case of local representations.

In Subsection \ref{balancedchekhovfock}, we describe the balanced Chekhov-Fock algebra $\Z\lambda$ associated to an ideal triagulation $\lambda$ of $\Sigma$ and characterize its irreducible representations. Then, in Subsection \ref{localrepresentationbalanced}, we introduce the local representations of $\Z\lambda$ and extend Main Theorem to decompose these local representations into irreducible factors. Finally, in Subsection \ref{skeinalgebra}, we use the previous decomposition to construct a family of representations of $\Sk{\overline\Sigma}$.

\subsection{Balanced Chekhov-Fock algebra}\label{balancedchekhovfock} Let $q$ be a primitive $N$-th root of unity with $N$ odd and let $\omega$ the unique fourth root of $q$ which is also a primitive $N$-th root of unity. Namely, if $q=e^{2i\pi\frac{k}{N}}$ with $N$ and $k$ co-prime, there is a unique $l\in\mathbb{Z}_4$ so that $k+lN\in4\mathbb{Z}$, then $\omega=e^{2i\pi\frac{k}{4N}}e^{i\frac{l\pi}{2}}$.

Let $\lambda$ be an ideal triangulation of $\Sigma$ whose edges are $(\lambda_1,...,\lambda_n)$. In order to avoid confusion, we will denote by $X_i$ the generators of $\T\lambda$ and by $Z_i$ the generators of $\mathcal{T}_\omega(\lambda)$.

A multi-index $\textbf{k}\in\mathbb{Z}^n$ is called \textit{$\lambda$-balanced} (or \textit{balanced}) if for each triangle of the triangulation whose edges are $j_1,j_2,j_3$ we have
$$k_{j_1}+k_{j_2}+k_{j_3}\in 2\mathbb{Z}.$$
A monomial $Z\in\mathcal{T}_\omega(\lambda)$ is \textit{balanced} if $Z$ is a scalar multiple of $Z_\textbf{k}$ where $\textbf{k}\in \mathbb{Z}^n$ is balanced (where $Z_\textbf{k}$ is defined as in Subsection \ref{chekhovfockalgebra}).

\begin{Def}
The \textit{balanced Chekhov-Fock algebra $\Z\lambda$} is the sub-algebra of $\mathcal{T}_\omega(\lambda)$ generated (as a vector space) by balanced monomials.
\end{Def}

In particular, the image of the map
$$\begin{array}{llll}
i: & \T\lambda & \longrightarrow & \mathcal{T}_\omega(\lambda) \\
& X_i & \longmapsto & Z_i^2
\end{array}$$
lies in $\Z\lambda$ so we will consider $\T\lambda$ as a sub-algebra of $\Z\lambda$.
 
The ideal triangulation $\lambda$ define canonically a train track $\tau_\lambda$ on $\Sigma$ which look like in Figure \ref{traintrack} on each triangle of the triangulation. Note that $\tau_\lambda$ has a switch on each edge of $\lambda$.

\begin{figure}[!h] 
\begin{center}
\includegraphics[height=4.5cm]{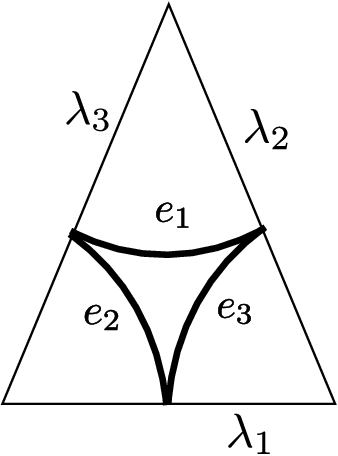}
\end{center}
\caption{Train track.} 
\label{traintrack}
\end{figure}

We denote by $\W$ the abelian group of integer weight systems on $\tau_\lambda$. Namely, an element $\alpha\in\W$ is a map that associates to each edge $e$ of $\tau_\lambda$ an integer $\alpha(e)$ in such a way that at each switch, the sum of the weights of the incoming edges equals the sum of the weights of the outgoing edges.

Given a weight system $\alpha\in\W$ and an edge $\lambda_i\in\lambda$, the sum of the weights of the edges incoming to $\lambda_i$ is an integer. It thus define a map 
$$\varphi:\W\longrightarrow \mathbb{Z}^n$$
whose image is exactly the set of balanced multi-index. Thus, given a integer weight system $\alpha\in\W$, we define $Z_\alpha\in\Z\lambda$ to be $Z_{\varphi(\alpha)}=[Z_1^{\alpha_1}...Z^{\alpha_n}_n]$ where $\varphi(\alpha)=(\alpha_1,...,\alpha_n)\in\mathbb{Z}^n$. In particular, one gets the following non-commutativity relations:
$$Z_\alpha Z_\beta = \omega^{4\Omega(\alpha,\beta)}Z_\beta Z_\alpha,$$
where $\Omega: \W \times \W \longrightarrow \mathbb{Z}$ is the Thurston intersection form (see \cite[Section 2]{bonahonwongII} for more details).

\begin{Def}
A \textit{twisted homomorphism} is a map $\zeta: \W \longrightarrow \mathbb{C}^*$ such that for every $\alpha,\beta\in\W$,
$$\zeta(\alpha+\beta)=(-1)^{\Omega(\alpha,\beta)}\zeta(\alpha)\zeta(\beta).$$
\end{Def}

Finally, note that each puncture $v_j$ defines a integer weight system $\eta_j\in\W$ as follow. Given an edge $e$ of $\tau_\lambda$, $\eta_j(e)\in\{0,1,2\}$ is the number of sides of $e$ that are adjacent to the same component of $\Sigma\setminus \tau_\lambda$ at $v_j$. Note that in particular, 
$$Z_{\eta_j}^2=i(P_j)$$
 where $P_j\in \T\lambda$ is the puncture invariant associated to $v_j$ and $i: \T\lambda \longrightarrow \Z\lambda$ is defined above.
 
Irreducible representations of $\Z\lambda$ were classified in \cite{bonahonwongIII}. They proved:

\begin{prop}[Bonahon-Wong \cite{bonahonwongII} Proposition 14] Up to isomorphism, an irreducible representation $\rho: \Z\lambda \longrightarrow \End(V)$ has dimension $N^{3g-3+s}$ and is classified by a twisted homomorphism $\zeta:\W\longrightarrow \mathbb{C}^*$ and a choice of $N$-th root $h_j=\zeta(\eta_j)^{1/N}$ for each puncture $v_j$. Such a representation satisfies:
\begin{itemize}
\item $\rho(Z_\alpha^N)=\zeta(\alpha)\Id_V$, for all $\alpha\in\W.$
\item $\rho(\eta_j)= \zeta(\eta_j)\Id_V$ for all $j\in\{1,...,s\}$.
\end{itemize}
\end{prop}

\subsection{Local representations of $\Z\lambda$}\label{localrepresentationbalanced} Here, we introduce the notion of local representation of the balanced Chekhov-Fock algebra $\Z\lambda$ associated to an ideal triangulation $\lambda$. We then extend Main Theorem  to give a decomposition of local representations of $\Z\lambda$ into its irreducible components.

In Subsection \ref{representationsChekhovFock}, we introduced the map
$$\mathcal{T}_\omega(\lambda) \longrightarrow \bigotimes_{T_i \in F(\lambda)}\mathcal{T}_\omega(T_i)$$
where $F(\lambda)$ is the set of faces of $\lambda$ and $\mathcal{T}_\omega(T_i)$ is the triangle algebra associated to the face $T_i$. It is clear that this map restricts to a morphism
$$\iota: \Z\lambda \longrightarrow \bigotimes_{T_i\in F(\lambda)} \Z{T_i}.$$
Here, $\Z{T_i}$ is the \textit{balanced triangle algebra} associated to the face $T_i$.

\begin{Def}
A \textit{local representation of the balanced Chekhov-Fock algebra $\Z\lambda$} is a representation $\rho: \Z\lambda \longrightarrow \End(V)$ that can be written as $(\rho_1\otimes...\otimes\rho_m)\circ \iota$ where each $\rho_i$ is an irreducible representation of $\Z{T_i}$.
\end{Def}

In order to classify local representations of $\Z\lambda$, we first have to understand the irreducible representations of the balanced triangle algebra $\Z T$. Let $\tau$ be the train track in $T$ with edges $e_1,e_2,e_3$ as in Figure \ref{traintrack} and denote by $\mathcal{W}(\tau,\mathbb{Z})$ The group  of integer weight systems on $\tau$. We have the following:

\begin{lemma}
Up to isomorphism, an irreducible representation of the balanced triangle algebra $\Z T$ has dimension $N$ and is classified by a twisted homomorphism $\zeta: \mathcal{W}(\tau,\mathbb{Z})\longrightarrow \mathbb{C}^*$ together with a choice of $N$-th root $C=\big(\zeta(\mu)\big)^{1/N}$ where $\mu\in \mathcal{W}(\tau,\mathbb{Z})$ is such that $\mu(e_i)=1$ for all $i\in\mathbb{Z}_3$. Such a representation satisfies
\begin{itemize}
\item $\rho(Z_\alpha^N)=\zeta(\alpha)\Id_V.$
\item $\rho(Z_\mu)=C\Id_V$.
\end{itemize}
\end{lemma}

\begin{proof}
The group $\mathcal{W}(\tau,\mathbb{Z})$ is generated by $\{\alpha_1,\alpha_2,\alpha_3\}$ where
$$\alpha_i(e_j)=\delta_{ij},~i,j\in\mathbb{Z}_3.$$
It follows that the balanced triangle algebra $\Z T$ is generated by $Z_{\alpha_1}^{\pm 1}, Z_{\alpha_2}^{\pm 1}$ and $Z_{\alpha_3}^{\pm 1}$ and the relations are
$$Z_{\alpha_1}Z_{\alpha_{i+1}}= \omega^{-2}Z_{\alpha_{i+1}}Z_{\alpha_i},~ i\in\mathbb{Z}_3.$$
If we denote by $Z_i$ the generator of $\mathcal{T}_\omega(T)$ associated to the edge $\lambda_i$ (so for instance $Z_{\alpha_1} = [Z_2Z_3]$), the map
$$\begin{array}{llll}
\Psi: & \Z T & \longrightarrow  & \mathcal{T}_{\omega}(T)\\
 & Z_{\alpha_i} & \longmapsto & Z_i^{-1}
 \end{array}$$
is an isomorphism of algebra such that $\Psi(\mu)=H^{-1}$ where $H=[Z_1Z_2Z_3]$.

In particular, an irreducible representation $\rho$ of $\Z \lambda$ has the form $\rho= \overline{\rho}\circ \Psi$ where $\overline\rho$ is an irreducible representation of $\mathcal{T}_\omega(T)$.

Using the result of Subsection \ref{representationsChekhovFock} and the fact that a twisted homomorphism of $\mathcal{W}(\tau,\mathbb{Z})$ is fully determined by its value on the $\alpha_i$, we get the result.
\end{proof}
Let $\tau_i$ be the restriction of the train track $\tau_\lambda$ to the triangle $T_i$ of the triangulation $\lambda$. A twisted homomorphism $\zeta: \W \longrightarrow \mathbb{C}^*$ induces a twisted homomorphism $\zeta_i: \mathcal{W}(\tau_{i},\mathbb{Z})\longrightarrow \mathbb{C}^*$ for each triangle $T_i$ of the triangulation $\lambda$. In particular, the following proposition is a straightforward extension of \cite[Proposition 6]{baibonahon}:

\begin{prop}
A local representation $\rho: \Z\lambda \longrightarrow \End(V)$ has dimension $N^{4g-4+2s}$ and is classified (up to isomorphism) by a twisted homomorphism $\zeta: \W \longrightarrow \mathbb{C}^*$ and a choice of $N$-th root $C=\zeta(\mu)^{1/N}$ where $\mu(e)=1$ for all edge $e$ of $\tau_\lambda$. Such a representation satisfies:
\begin{itemize}
\item $\rho(Z_\alpha^N)=\zeta(\alpha)\Id_V$.
\item $\rho(Z_\mu)=C\Id_V$.
\end{itemize}
\end{prop}

\noindent Finally, Main Theorem implies the following:

\begin{theo}\label{theolocalbalanced}
Let $\rho: \Z\lambda \longrightarrow \End(V)$ be the (isomorphism class of) representation classified by the twisted homomorphism $\zeta: \W \longrightarrow \mathbb{C}^*$ and the choice of $N$-th root $C= \zeta(\mu)^{1/N}$ (where $\mu$ is defined as above). Then $\rho=\bigoplus_{i\in \mathcal{I}} \rho^{(i)}$ where each $\rho^{(i)}$ is irreducible, classified by the same twisted homomorphism $\zeta$ and $N$-th root $h_k^{(i)}=\big(\zeta(\eta_k)\big)^{1/N}$ with $h_1^{(i)}...h_s^{(i)}=C$ (here, the $\eta_k$ are defined as in Subsection \ref{balancedchekhovfock}).

Moreover, for each choice of $N$-th root $h_k=\big(\zeta(\eta_k) \big)^{1/N}$ for each $k\in\{1,...,s\}$, there are exactly $N^g$ indices $i\in I$ such that $h^{(i)}_k=h_k$ for all $k$.
\end{theo}

\begin{proof}
Let $\rho$ be a local representation of $\Z\lambda$ classified by $\zeta$ and $C$. In particular, $\rho$ induces a local representation $\overline\rho:=\rho \circ i$ where $i: \T\lambda \hookrightarrow \Z\lambda$. Note that the local representation $\overline\rho$ is classified by the weight $\zeta(\beta_i)$ for all edge $\lambda_i$, where $Z_{\beta_i}=Z_i^2= i(X_i)$.

Let $P_j\in \T\lambda$ be the puncture invariant associated to the puncture $v_j$. Note that, the image of $P_j$ in $\Z\lambda$ is $Z_{\eta_j}^2$. We claim that the eigenspaces of $\overline{\rho}(P_j)$ correspond to the eigenspaces of $\rho(Z_{\eta_j})$. In fact, if $V_{h_j}(Z_{\eta_j})$ is the eigenspace of $\rho(Z_{\eta_j})$ corresponding to the eigenvalue $h_j=\big(\zeta(\eta_j) \big)^{1/N}$, then one has the inclusion
$$V_{h_j}(Z_{\eta_j})\subset V_{p_j}(P_j),$$
where $V_{p_j}(P_j)$ is the eigenspace of $\overline{\rho}(P_j)=\rho(Z_{\eta_j}^2)$ corresponding to the eigenvalue $p_j=h_j^2$. Because there are only $N$ different possible eigenvalues of $\rho(Z_{\eta_j})$, a dimension argument shows the equality.

Now, we can apply Main Theorem and we get that, for each choice of $(h_1,...,h_s)$ where $h_j = \big(\zeta(Z_{\eta_j})\big)^{1/N}$, the intersection $V_{h_1}(Z_{\eta_1})\cap ... \cap V_{h_s}(Z_{\eta_s})$ has dimension$N^{4g-3+s}$ and so is made of $N^g$ copies of the irreducible representation of $\Z\lambda$ classified by $\zeta$ and $h_1,...,h_s$.
\end{proof}

In \cite[Section 3.]{bonahonwongII}, the author associate a character $r_\zeta \in \chi\big(\Sigma, \text{SL}(2,\mathbb{C})\big)$ to a twisted homomorphism $\zeta: \W \longrightarrow \mathbb{C}^*$. In particular, the (irreducible or local) representations of the balanced Chekhov-Fock algebra associated to an ideal triangulation $\lambda$ of $\Sigma$ are classified, up to finitely many choice, by a Zariski open set in $\chi\big(\Sigma, \text{SL}(2,\mathbb{C})\big)$.

Note that, if $r_\zeta$ is the character associated to the twisted homomorphism $\zeta$, the holonomy of $r_\zeta$ around a puncture $v_j$ is parabolic exactly when $\zeta(\eta_j)=1$.

\subsection{Representations of $\Sk{\overline{\Sigma}}$}\label{skeinalgebra} We explain here how Theorem \ref{theolocalbalanced} gives rise to a new family of representations of the Kauffman bracket skein algebra of the closed surface $\overline\Sigma= \Sigma \cup\{v_1,...,v_s\}$.

\medskip

\noindent\textbf{Skein algebra.} Given a 3-manifold $M$, and a non-zero complex number $A\in\mathbb{C}^*$, consider the complex vector space $V(M)$ freely generated by isotopy classes of framed links in $M$. The \textit{skein module} $\Sk M$ of $M$ is the quotient of $V(M)$ by the \textit{Kauffman bracket skein relations} as defined in Figure \ref{skeinrelations2}.

\begin{figure}[!h] 
\begin{center}
\includegraphics[height=5cm]{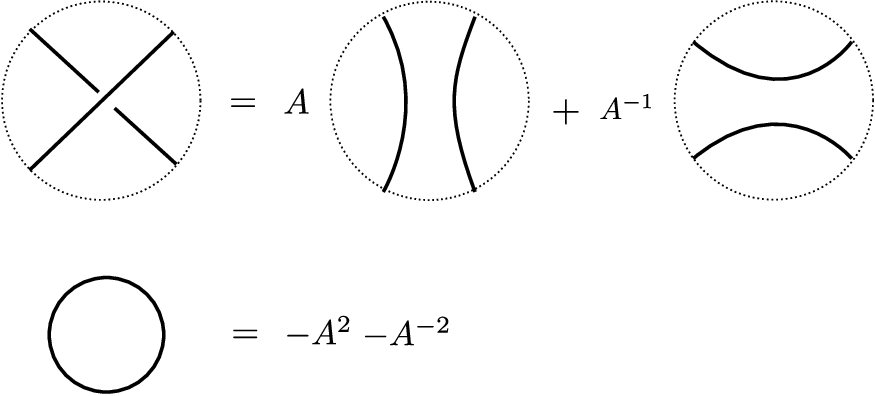}
\end{center}
\caption{Kauffman bracket skein relations} 
\label{skeinrelations2}
\end{figure}

Namely, we identify three different links when differ by the previous relation in an open ball and agree everywhere else.

Given a framed link $K\subset M$, we denote by $[K]$ its image in the skein module $\Sk{M}$.

When $M=S\times [0,1]$ for a surface $S$, the skein module $\Sk{M}=\Sk S$ inherits an algebra structure given by superposition of links. Namely, given two framed links $K_1$ and $K_2$ in $S\times [0,1]$, the product $[K_1].[K_2]$ is defined to be the image of $K_1\cup K_2$ in $\Sk S$, where $K_1\cup K_2$ is given by the superposition of $K_1$ on top of $K_2$ where we rescaled so that $K_1\subset S\times [0,\frac{1}{2}]$ and $K_2\subset S\times [\frac{1}{2},1]$. We call $\Sk S$ the \textit{Kauffman bracket skein algebra of $S$}.

Finite dimensional representations of the skein algebra $\Sk{S}$ are of main importance as they appear naturally in Topological Quantum Field Theory (TQFT) as for example, in the Witten-Reshetekin-Turaev TQFT (see for example \cite{bhmv,turaevbook}).

\medskip

\noindent\textbf{Classical shadow and quantum trace.} Let $\mu: \Sk \Sigma \longrightarrow \End (V)$ be an irreducible representation of the Kauffman bracket skein algebra of $\Sigma$.

In \cite{bonahonwong1} (see also \cite{thangle} for a simpler proof), the authors proved that, if $A$ is a primitive $N$-th root of $-1$, the $N$-th Chebyshev polynomial $T_N$ of the first kind of any skein $[K]\in \Sk \Sigma$ is a central element in $\Sk\Sigma$. In particular, the pre-composition of $\rho$ by $T_N$ maps each skein $[K]$ to a multiple of the identity in $\End(V)$. This multiple of the identity can be interpreted as an element $r_\mu \in \chi \big(\Sigma,\text{SL}_2(\mathbb{C})\big)$ in the $\text{SL}(2,\mathbb{C})$ character variety of $\Sigma$. This character is called \textit{the classical shadow} of the representation $\mu$.

When $A=\omega^{-2}$ (so $A$ is a primitive $N$-th root of $-1$) and $\lambda$ is an ideal triangulation of $\Sigma$, the authors of \cite{quantumtrace} (see also \cite{thangle2} for a more conceptual proof) constructed a \textit{quantum trace map}
$$\text{tr}_\omega^\lambda: \Sk\Sigma \longrightarrow \Z\lambda,$$
which turns out to be an injective algebra homomorphism.

In particular, by pre-composing irreducible representations of $\Z\lambda$ by the quantum trace, Bonahon and Wong obtained in \cite{bonahonwongII} a family of irreducible representations of the Kauffman bracket skein algebra of $S$ indexed by the character variety $\chi \big(\Sigma,\text{SL}_2(\mathbb{C})\big)$. Moreover, taking the classical shadow of such an irreducible representation recovers the character.

\medskip

\noindent\textbf{Representations of $\Sk{\overline{\Sigma}}$.} The inclusion $\Sigma \hookrightarrow \overline\Sigma$ gives an algebra homomorphism
$$\iota: \Sk{\overline\Sigma} \longrightarrow \Sk\Sigma.$$
Let $r\in \chi\big(\Sigma,\text{SL}(2,\mathbb{C})\big)$ be a character obtain from a character $r' \in \chi\big(\overline{\Sigma},\text{SL}(2,\mathbb{C})\big)$ (namely, the holonomy of $r$ around each puncture is trivial). If $\zeta: \W \longrightarrow \mathbb{C}^*$ is the twisted homomorphism associated to $r$, $\zeta(\eta_j)=1$ for each puncture $v_j$.

Denote by 
$$\rho :\Z\lambda \longrightarrow \End(V)$$
the local representation of $\Z\lambda$ classified by $\zeta$ and the $N$-th root $C= (-\omega^4)^s$. Let $E\subset V$ the intersection of the eigenspaces of $\rho(Z_{\eta_k})$ for $k\in\{1,...,s\}$ corresponding the the eigenvalue $-\omega^4$.

By Theorem \ref{theolocalbalanced}, the vector space $E$ is stable by $\rho\big(\Z\lambda\big)$, so we get an induced representation  $\rho': \Z\lambda \longrightarrow \End(E)$. Note that, $\rho'$ is made of $N^g$ copies of the irreducible representation of $\Z\lambda$ classified by $\zeta$ and puncture invariant $-\omega^4$.

\begin{prop}
There is a proper linear subspace $F\subset E$ such that the composition
$$\mu: \Sk{\overline\Sigma} \overset{\iota}{\longrightarrow} \Sk\Sigma \overset{\overline{\rho}}{\longrightarrow} \End(E)$$
induces a representation of $\Sk{\overline\Sigma}$. The classical shadow of each irreducible factor of $\mu$ is same. Finally, the dimension of $F$ is at least $N^{4g-3}$ when $g>1$ and at least $N^2$ when $g=1$.
\end{prop}

\begin{proof}
This is a direct consequence of the construction of \cite{bonahonwongIII}. In fact, using the decomposition of $\overline\rho$ into irreducible parts and considering the \textit{total off-diagonal kernel} of each irreducible factor (see \cite[Section 4.2]{bonahonwongIII}), one gets the result.
\end{proof}

Note that, the vector space $F$ is canonically associated to the triangulation $\lambda$, which make the family of representations described above easier to handle for computations.

\bibliography{local representations.bbl}
\bibliographystyle{alpha}

\end{document}